\documentclass[12pt,fleqn]{article}

\usepackage{amsmath}
\usepackage{amsthm}
\usepackage{amssymb}
\usepackage{amsfonts}
\usepackage{epsf}
\usepackage{graphicx}
\usepackage{epstopdf}
\usepackage{color}
\usepackage{mathptmx}

\newtheorem{theorem}{Theorem}[section]
\newtheorem{lemma}[theorem]{Lemma}

\theoremstyle{definition}

\newtheoremstyle{named}{}{}{\itshape}{}{\bfseries}{.}{.5em}{\thmnote{#3's }#1} \theoremstyle{named} 

\theoremstyle{remark}

\numberwithin{equation}{section}

\numberwithin{equation}{section}

\title{Existence of a lower bound for the distance between point masses of relative equilibria in $\mathbb{S}^{k-1}$, $k\geq 3$. }
\author{Pieter Tibboel$^\ast$}
\begin{document}
\maketitle
\begin{abstract}
  We prove that if for the curved $n$-body problem in $\mathbb{S}^{k-1}$, $k\geq 3$, the masses are given, the minimum distance between the point masses of a specific type of relative equilibrium solution that is a generalisation of positive elliptic relative equilibria and positive elliptic-elliptic relative equilibria has a universal lower bound that is not equal to zero.
\end{abstract}
\begin{description}

\item \hspace*{3.8mm}$\ast$ Department of Mathematics, City University of
Hong Kong, Hong Kong. \\
Email: \texttt{ptibboel@cityu.edu.hk}

\end{description}

\newpage

\section{Introduction. }
  By $n$-body problems we mean problems where the goal is to find the motion of $n$ point particles. Relative equilibria are solutions to such problems where the point particles represent rotating configurations of fixed size and shape. The Newtonian $n$-body problem is the problem of finding the dynamics of $n$ point particles in Euclidean space under the influence of Newton's gravitational law and the \textit{$n$-body problem in spaces of constant curvature} is an extension of that $n$-body problem into spaces of nonzero, constant Gaussian curvature, which means that the space is either spherical (if the curvature is positive), or hyperbolical (if the curvature is negative) (see \cite{DPS1}, \cite{DPS2} and \cite{DPS3}).

Research into $n$-body problems for spaces of constant Gaussian curvature goes back as far as the 1830s, when Bolyai and Lobachevsky (see \cite{BB} and \cite{Lo} respectively) independently proposed a curved 2-body problem in hyperbolic space $\mathbb{H}^{3}$. Since then, $n$-body problems in spaces of constant Gaussian curvature have been investigated by mathematicians such as Dirichlet, Schering (see \cite{S1}, \cite{S2}), Killing (see \cite{K1}, \cite{K2}, \cite{K3}), Liebmann (see \cite{L1}, \cite{L2}, \cite{L3}) and Kozlov and Harin (see \cite{KH}). However, the succesful study of $n$-body problems in spaces of constant Gaussian curvature for the case that $n\geq 2$ began with \cite{DPS1}, \cite{DPS2}, \cite{DPS3}  by Diacu, P\'erez-Chavela and Santoprete. After this breakthrough, further results for the $n\geq 2$ case were then obtained by Cari\~nena, Ra\~nada, Santander (see \cite{CRS}), Diacu (see \cite{D1}, \cite{D2}, \cite{D3}, \cite{D4}, \cite{D5}), Diacu, Kordlou (see \cite{DK}), Diacu, P\'erez-Chavela (see \cite{DP}), Diacu, Thorn (see \cite{DT}) and Tibboel (see \cite{T}, \cite{T2}, \cite{T3}). For a more detailed historical overview, please see \cite{D2}, \cite{D3}, \cite{D4}, \cite{D6}, \cite{DK}, or \cite{DPS1}.

For the Newtonian $n$-body problem, M. Shub proved that if the masses of the $n$-body problem are known, then the set of all equivalence classes of relative equilibria (an equivalence class here being all relative equilibria that are equivalent under rotation and scalar multiplication) that exist for those masses is compact. As a direct consequence, Shub proved that there exists a universal nonzero lower bound for the distance between the point particles of each such a relative equilibrium if both the masses and the angular velocity (i.e. the speed at which a relative equilibrium rotates) are fixed (see \cite{Shub}). Shub's results were a potential first step in what may lead to a proof of the famous sixth Smale problem (see \cite{Smale}) which states a stronger result, namely that such sets of equivalence classes of relative equilibria are in fact, finite.

In \cite{T3}, Shub's results were proven for a class of relative equilibria that include all relative equilibria in the spaces $\mathbb{S}^{2}$ and $\mathbb{H}^{2}$ and a subclass of the positive elliptic relative equilibria in $\mathbb{S}^{3}$ as defined in \cite{D3} and a subclass of the negative elliptic relative equilibria in $\mathbb{H}^{3}$ as defined in \cite{D3}. In this paper, we will prove Shub's results for a generalisation of the classes of positive elliptic relative equilibria and positive elliptic-elliptic relative equilibria (see \cite{D3}) in $\mathbb{S}^{3}$ and as a consequence prove Shub's results for all classes of relative equilibria in $\mathbb{S}^{3}$. To that extent, we will first formulate the system of differential equations that describes the $n$-body problem we will consider in this paper, after which we will specify our main theorem:

In \cite{D4} and \cite{DK} it was shown that the $n$-body problem in spaces of constant curvature is equivalent with the $n$-body problem in spaces of constant curvature for which it is assumed that the curvature is equal to either $+1$, or $-1$. The $n$-body problem in spaces of constant curvature can then be described, defining the space
  \begin{align*}
    \mathbb{M}_{\sigma}^{k-1}=\{(x_{1},....,x_{k})\in\mathbb{R}^{k}|x_{1}^{2}+...+x_{k-1}^{2}+\sigma x_{k}^{2}=\sigma\},
  \end{align*}
  where $\sigma$ equals either $+1$, or $-1$
  and for $x$, $y\in\mathbb{M}_{\sigma}^{k}$
  \begin{align*}
    x\odot_{k} y=x_{1}y_{1}+...+x_{k-1}y_{k-1}+\sigma x_{k}y_{k},
  \end{align*}
  following \cite{DPS1}, \cite{DPS2}, \cite{DPS3}, \cite{D3} and \cite{DK},
  as the system of differential equations
  \begin{align}\label{EquationsOfMotion}
   \ddot{q}_{i}=\sum\limits_{j=1,\textrm{ }j\neq i}^{n}\frac{m_{j}(q_{j}-\sigma(q_{i}\odot_{k} q_{j})q_{i})}{(\sigma-\sigma(q_{i}\odot_{k} q_{j})^{2})^{\frac{3}{2}}}-\sigma(\dot{q}_{i}\odot_{k}\dot{q}_{i})q_{i},\textrm{ }i\in\{1,...,n\},
  \end{align}
  where $q_{1}$,..., $q_{n}\in\mathbb{M}_{\sigma}^{k-1}$ are the coordinate vectors of the $n$ point particles of the problem with respective masses $m_{1}$,...,$m_{n}$, $k\geq 2$. However, as we are considering relative equilibria on $\mathbb{S}^{k-1}$, we may assume that $\sigma=1$ and thus the inner product becomes the Euclidean inner product $\langle\cdot,\cdot\rangle_{k}$ on $\mathbb{R}^{k}$. \\
  With our $n$-body problem in place, our next step is to formulate the class of relative equilibria we will use in this paper:

  Let
  \begin{align*}
    T(t)=\begin{pmatrix}
      \cos{t} & -\sin{t} \\
      \sin{t} & \cos{t}
    \end{pmatrix}
  \end{align*}
and let $I$ be the $2\times 2$ identity matrix. \\
Let $p\in\mathbb{N}$ and $A=(A_{1},...,A_{p})\in\mathbb{R}^{p}$. Let
\begin{align*}
  T_{k}(At)=\begin{pmatrix}
    T(A_{1}t)&\hdots & 0\\
    \vdots & \ddots & 0 \\
    0&\hdots & T(A_{p}t)
  \end{pmatrix}\textrm{ if }k=2p\\ \textrm{ and }T_{k}(At)=\begin{pmatrix}
    T(A_{1}t)&\hdots & 0 & 0\\
    \vdots & \ddots & \vdots &\vdots\\
    0 & \hdots & T(A_{p}t) & 0\\
    0 & \hdots & 0 & 1
  \end{pmatrix}\textrm{ if }k=2p+1.
\end{align*}
Then if $Q_{1}$,...,$Q_{n}\in\mathbb{S}^{k-1}$, we call any solution $q_{1}(t)=T_{k}(At)Q_{1}$,...,$q_{n}(t)=T_{k}(At)Q_{n}$ of (\ref{EquationsOfMotion}) a \textit{relative equilibrium}. Let $\|\cdot\|_{k}$ be the Euclidean norm on $\mathbb{R}^{k}$. We will prove that
\begin{theorem}\label{Main circle theorem}
  There exists a $c>0$ such that $\|Q_{i}-Q_{j}\|_{k}>c$ for all $i$, $j\in\{1,...,n\}$, $i\neq j$, for all relative equilibrium solutions $q_{1}(t)=T_{k}(At)Q_{1}$,...,$q_{n}(t)=T_{k}(At)Q_{n}$ of (\ref{EquationsOfMotion}) if $A$ and $m_{1}$,...,$m_{n}$ are given.
\end{theorem}
Before we can prove Theorem~\ref{Main circle theorem}, we will need to introduce further terminology and prove a crucial lemma, which will be done in section~\ref{Lemma section}. Theorem~\ref{Main circle theorem} will then subsequently be proved in section~\ref{Proof of Main Theorem}.
\section{Background theory}\label{Lemma section}
In this section we will introduce further terminology and a lemma needed to prove Theorem~\ref{Main circle theorem}:

Define for $A=(A_{1},...,A_{p})\in\mathbb{R}^{p}$
\begin{align*}
    \mathbf{A}_{k}=\begin{pmatrix}
      A_{1} & 0 & \hdots & 0 & 0 \\
      0 & A_{1} & \hdots & 0 & 0 \\
      \vdots & \vdots & \ddots & \vdots & \vdots \\
      0 & 0 & \hdots & A_{p} & 0\\
      0 & 0 & \hdots & 0 & A_{p}\\
    \end{pmatrix}\textrm{ if }k=2p\textrm{ and }\mathbf{A}_{k}=\begin{pmatrix}
      A_{1} & 0 & \hdots & 0 & 0 & 0 \\
      0 & A_{1} & \hdots & 0 & 0 & 0\\
      \vdots & \vdots & \ddots & \vdots & \vdots & \vdots\\
      0 & 0 & \hdots & A_{p} & 0 & 0\\
      0 & 0 & \hdots & 0 & A_{p} & 0\\
      0 & 0 & \hdots & 0 & 0 & 0
    \end{pmatrix}\textrm{ if }k=2p+1.
  \end{align*}
Then
\begin{lemma}\label{Lemma for circle}
  If $q_{1}(t)=T_{k}(At)Q_{1}$,...,$q_{n}(t)=T_{k}(At)Q_{n}$ is a relative equilibrium solution of (\ref{EquationsOfMotion}), then
  \begin{align}\label{Lemma Power}
   -\mathbf{A}^{2}Q_{i}=\sum\limits_{j=1,\textrm{ }j\neq i}^{n}\frac{m_{j}(Q_{j}-\langle Q_{i},Q_{j}\rangle Q_{i})}{(1-\langle Q_{i},Q_{j}\rangle^{2})^{\frac{3}{2}}}-\|\mathbf{A}Q_{i}\|_{k}^{2}Q_{i},\textrm{ }i\in\{1,...,n\},
\end{align}
\end{lemma}
\begin{proof}
  Inserting $q_{i}(t)=T_{k}(At)Q_{i}$, $q_{j}(t)=T_{k}(At)Q_{j}$ into (\ref{EquationsOfMotion})
  and using that $(T_{k}(At))''=-\mathbf{A}^{2}T_{k}(At)$ and $(T(A_{l}t))'=A_{l}T(A_{l}t)$, $l\in\{1,...,p\}$ and that for $x$, $y\in\mathbb{R}^{k}$ $\langle T_{k}(At)x,T_{k}(At)y\rangle=\langle x,y\rangle$ gives
  \begin{align}\label{EquationsOfMotion2}
   -\mathbf{A}^{2}T_{k}(At)Q_{i}=\sum\limits_{j=1,\textrm{ }j\neq i}^{n}\frac{m_{j}(T_{k}(At)Q_{j}-\langle Q_{i},Q_{j}\rangle T_{k}(At)Q_{i})}{(1-\langle Q_{i},Q_{j}\rangle^{2})^{\frac{3}{2}}}-\|\mathbf{A}Q_{i}\|_{k}^{2}T_{k}(At)Q_{i},\textrm{ }i\in\{1,...,n\}.
\end{align}
Multiplying both sides of (\ref{EquationsOfMotion2}) from the left with $T(-At)$ then gives (\ref{Lemma Power}).
\end{proof}
\section{Proof of Theorem~\ref{Main circle theorem}}\label{Proof of Main Theorem}
\begin{proof}
  Assume the contrary to be true. Then there exist relative equilibrium solutions $q_{1p}=T_{k}(At)Q_{1p}$,...,$q_{np}=T_{k}(At)Q_{np}$, $p\in\mathbb{N}$, for which there exist $i_{1}$, $i_{2}\in\{1,...,n\}$, $i_{1}$, $i_{2}$ such that $\lim\limits_{p\rightarrow\infty}\|Q_{i_{1}p}-Q_{i_{2}p}\|_{k}=0$.
  Renumbering the $\{Q_{ip}\}_{p=1}^{\infty}$ in terms of $i$ if necessary, we may assume that $\lim\limits_{p\rightarrow\infty}\|Q_{ip}-Q_{jp}\|_{k}=0$ for $i$, $j\in\{1,...,l\}$, $l\leq n$ and $\lim\limits_{p\rightarrow\infty}\|Q_{ip}-Q_{jp}\|_{k}\neq 0$ if $i\in\{1,...,l\}$ and $j\notin\{1,...,l\}$.
  Because of Lemma~\ref{Lemma Power}, we have that
  \begin{align*}
     -\mathbf{A}^{2}Q_{1p}=\sum\limits_{j=2}^{n}\frac{m_{j}(Q_{jp}-\langle Q_{1p},Q_{jp}\rangle Q_{1p})}{(1-\langle Q_{1p},Q_{jp}\rangle^{2})^{\frac{3}{2}}}-\|\mathbf{A}Q_{1p}\|_{k}^{2}Q_{1p},
  \end{align*}
  which can be rewritten as
  \begin{align}\label{EquationsOfMotion4}
     -\mathbf{A}^{2}Q_{1p}=\sum\limits_{j=2}^{l}\frac{m_{j}(Q_{jp}-\langle Q_{1p},Q_{jp}\rangle Q_{1p})}{(1-\langle Q_{1p},Q_{jp}\rangle^{2})^{\frac{3}{2}}}+\sum\limits_{j=l+1}^{n}\frac{m_{j}(Q_{jp}-\langle Q_{1p},Q_{jp}\rangle Q_{1p})}{(1-\langle Q_{1p},Q_{jp}\rangle^{2})^{\frac{3}{2}}}-\|\mathbf{A}Q_{1p}\|_{k}^{2}Q_{1p}.
  \end{align}
  As $Q_{1p}\in\mathbb{S}^{k-1}$, there exists a rotation matrix $R_{p}$ such that \begin{align*}R_{p}Q_{1p}=\begin{pmatrix}
    1\\
    0\\
    \vdots\\
    0
  \end{pmatrix}=:e_{1}\in\mathbb{R}^{k}.
  \end{align*}
  Multiplying both sides of (\ref{EquationsOfMotion4}) with $R_{p}$ then gives, as $R_{p}$ is a rotation matrix,
  \begin{align}\label{EquationsOfMotion5}
     -R_{p}\mathbf{A}^{2}R_{p}^{-1}e_{1}&=\sum\limits_{j=2}^{l}\frac{m_{j}(R_{p}Q_{jp}-\langle e_{1},R_{p}Q_{jp}\rangle e_{1})}{(1-\langle e_{1},R_{p}Q_{jp}\rangle^{2})^{\frac{3}{2}}}\\
     &+\sum\limits_{j=l+1}^{n}\frac{m_{j}(R_{p}Q_{jp}-\langle e_{1},R_{p}Q_{jp}\rangle e_{1})}{(1-\langle e_{1},R_{p}Q_{jp}\rangle^{2})^{\frac{3}{2}}}-\|\mathbf{A}R_{p}^{-1}e_{1}\|_{k}^{2}e_{1}.\nonumber
  \end{align}
  We will write for notational convenience
  \begin{align*}
    R_{p}Q_{jp}=:V_{jp}=\begin{pmatrix}
      V_{jp1}\\
      \vdots\\
      V_{jpk}
    \end{pmatrix}\textrm{, }\widehat{V}_{jp}=\begin{pmatrix}
      V_{jp2}\\
      \vdots\\
      V_{jpk}
    \end{pmatrix}\in\mathbb{R}^{k-1}
  \end{align*}
  and let $W_{p}\in\mathbb{R}^{k-1}$ be the $k-1$-dimensional vector that consists of the second until the $k$th coordinate of $-R_{p}\mathbf{A}^{2}R_{p}^{-1}e_{1}$. \\
  Leaving out the identity for the first coordinate of the vectors in (\ref{EquationsOfMotion5}), we then get, as all coordinates of $e_{1}$ save the first one are equal to zero,
  \begin{align}\label{EquationsOfMotion6}
    W_{p}=\sum\limits_{j=2}^{l}\frac{m_{j}\widehat{V}_{jp}}{(1-V_{jp1}^{2})^{\frac{3}{2}}}+\sum\limits_{j=l+1}^{n}\frac{m_{j}\widehat{V}_{jp}}{(1-V_{jp1}^{2})^{\frac{3}{2}}}.
  \end{align}
  Note that by construction $\|W_{p}\|_{k-1}\leq\|e_{1}\|_{k}=1$ and that
  \begin{align*}
    \left\|\sum\limits_{j=l+1}^{n}\frac{m_{j}\widehat{V}_{jp}}{(1-V_{jp1}^{2})^{\frac{3}{2}}}\right\|_{k-1}
  \end{align*}
  is bounded, because $\|\widehat{V}_{jp}\|_{k-1}\leq 1$ and for $j\geq l+1$ $1-V_{jp1}^{2}$ does not go to zero, because $V_{jp}-e_{1}$ does not go to zero for $j\geq l+1$. Thus, subtracting from both sides of (\ref{EquationsOfMotion6}) the term
  \begin{align*}
    \sum\limits_{j=l+1}^{n}\frac{m_{j}\widehat{V}_{jp}}{(1-V_{jp1}^{2})^{\frac{3}{2}}}
  \end{align*}
  we get
  \begin{align}\label{EquationsOfMotion7}
    W_{p}-\sum\limits_{j=l+1}^{n}\frac{m_{j}\widehat{V}_{jp}}{(1-V_{jp1}^{2})^{\frac{3}{2}}}=\sum\limits_{j=2}^{l}\frac{m_{j}\widehat{V}_{jp}}{(1-V_{jp1}^{2})^{\frac{3}{2}}}.
  \end{align}
  If we now take the squared norms of both sides of (\ref{EquationsOfMotion7}), we get
  \begin{align}\label{EquationsOfMotion8}
    \left\|W_{p}-\sum\limits_{j=l+1}^{n}\frac{m_{j}\widehat{V}_{jp}}{(1-V_{jp1}^{2})^{\frac{3}{2}}}\right\|_{k-1}^{2}=\left\|\sum\limits_{j=2}^{l}\frac{m_{j}\widehat{V}_{jp}}{(1-V_{jp1}^{2})^{\frac{3}{2}}}\right\|_{k-1}^{2}=\sum\limits_{i=2}^{l}\sum\limits_{j=2}^{l}\frac{m_{i}m_{j}\langle\widehat{V}_{ip},\widehat{V}_{jp}\rangle}{(1-V_{ip1}^{2})^{\frac{3}{2}}(1-V_{jp1}^{2})^{\frac{3}{2}}}.
  \end{align}
  If we then let $\alpha_{ijp}$ be the angle between $\widehat{V}_{ip}$ and $\widehat{V}_{jp}$, (\ref{EquationsOfMotion8}) gives
  \begin{align}\label{EquationsOfMotion9}
    \left\|W_{p}-\sum\limits_{j=l+1}^{n}\frac{m_{j}\widehat{V}_{jp}}{(1-V_{jp1}^{2})^{\frac{3}{2}}}\right\|_{k-1}^{2}=\sum\limits_{i=2}^{l}\sum\limits_{j=2}^{l}\frac{m_{i}m_{j}\|\widehat{V}_{ip}\|_{k-1}\|\widehat{V}_{jp}\|_{k-1}\cos{\alpha_{ijp}}}{(1-V_{ip1}^{2})^{\frac{3}{2}}(1-V_{jp1}^{2})^{\frac{3}{2}}}.
  \end{align}
  Note that $\|\widehat{V}_{jp}\|_{k-1}^{2}=V_{jp2}^{2}+...+V_{jpk}^{2}=1-V_{jp1}^{2}$ and that for $p$ large enough, $\cos{\alpha_{ijp}}>\frac{1}{2}$. Combining these facts with (\ref{EquationsOfMotion9}) then means
  \begin{align}\label{EquationsOfMotion10}
    \left\|W_{p}-\sum\limits_{j=l+1}^{n}\frac{m_{j}\widehat{V}_{jp}}{(1-V_{jp1}^{2})^{\frac{3}{2}}}\right\|_{k-1}^{2}\geq\frac{1}{2}\sum\limits_{i=2}^{l}\sum\limits_{j=2}^{l}\frac{m_{i}m_{j}}{(1-V_{ip1}^{2})(1-V_{jp1}^{2})}.
  \end{align}
  Now the left-hand side of (\ref{EquationsOfMotion10}) is bounded, yet the right-hand side of (\ref{EquationsOfMotion10}) is unbounded. This is a contradiction and completes our proof.
\end{proof}

\end{document}